\documentclass[11pt]{article}
\usepackage{amsmath,amsfonts,amssymb,latexsym,amsthm,dsfont}
\usepackage{geometry,graphicx}
\usepackage[backref]{hyperref}

\title{On the long time behavior of the TCP window size process}

\author{%
  Djalil~\textsc{Chafa\"\i}, %
  Florent~\textsc{Malrieu}, %
  Katy~\textsc{Paroux}}

\date{Preprint 2009}

\newcommand{\ind}{\mathds{1}}
\newtheorem{thm}{Theorem}[section]%
\newtheorem{cor}[thm]{Corollary}%
\newtheorem{lem}[thm]{Lemma}%

\newtheorem{defi}[thm]{Definition}%
\newtheorem{rem}[thm]{Remark}%
\newcommand{\dE}{\mathbb{E}}

\newcommand{\dN}{\mathbb{N}}
\newcommand{\dP}{\mathbb{P}}
\newcommand{\dR}{\mathbb{R}}

\newcommand{\cL}{\mathcal{L}}
\newcommand{\cN}{\mathcal{N}}

\newcommand{\ABS}[1]{{{\left| #1 \right|}}} 
\newcommand{\BRA}[1]{{{\left\{#1\right\}}}} 
\newcommand{\NRM}[1]{{{\left\| #1\right\|}}} 
\newcommand{\PAR}[1]{{{\left(#1\right)}}} 
\newcommand{\SBRA}[1]{{{\left[#1\right]}}} 
\renewcommand{\leq}{\leqslant}
\renewcommand{\geq}{\geqslant}

\begin{document}

\maketitle

\begin{abstract}
  The TCP window size process appears in the modeling of the famous
  Transmission Control Protocol used for data transmission over the Internet.
  This continuous time Markov process takes its values in $[0,\infty)$, is
  ergodic and irreversible. It belongs to the Additive Increase Multiplicative
  Decrease class of processes. The sample paths are piecewise linear
  deterministic and the whole randomness of the dynamics comes from the jump
  mechanism. Several aspects of this process have already been investigated in
  the literature. In the present paper, we mainly get quantitative estimates
  for the convergence to equilibrium, in terms of the $W_1$
  Wasserstein coupling distance, for the process and also for its embedded
  chain.
\end{abstract}

{\footnotesize %

\medskip

\noindent\textbf{AMS-MSC.}  68M12 ; 60K30 ; 60K25 ; 90B18

}

\section{Introduction}

The TCP protocol is one of the main data transmission protocols of the
Internet. It has been designed to adapt to the various traffic conditions of
the actual network. For a connection, the maximum number of packets that can
be sent at each round is given by a variable $W$, called the \emph{congestion
  window size}. If all the $W$ packets are successfully transmitted, then $W$
is increased by $1$, otherwise it is multiplied by $\delta\in[0,1)$ (detection
of a congestion). As shown in \cite{DGR,GRZ,ott-kemperman-mathis}, a correct
scaling of this process leads to a continuous time Markov process, called
general TCP window size process. This process $X={(X_t)}_{t\geq0}$ has
$[0,\infty)$ as state space and its infinitesimal generator is given, for any
smooth function $f:[0,\infty)\to\dR$, by
\begin{equation}\label{eq:G1}
  L(f)(x) = f'(x) + x\int_0^1(f(h x)-f(x))H(dh)
\end{equation}
for some probability measure $H$ supported in $[0,1)$. This window size
${(X_t)}_t$ increases linearly (this is the $f'$ part of $L$) until the
reception of a congestion signal which forces the reduction of the window size
by a multiplicative factor of law $H$ or equal to $\delta$ in the simplest
case (this is the jump part of $L$). The sample paths of $X$ are deterministic
between jumps, the jumps are multiplicative, and the whole randomness of the
dynamics relies on the jump mechanism. Of course, the randomness of $X$ may
also come from a random initial value. The process ${(X_t)}_{t\geq0}$ appears
as an Additive Increase Multiplicative Decrease process (AIMD), but also as a
very special Piecewise Deterministic Markov Process (PDMP) initially introduced in \cite{davis}. In this direction,
\cite{maulik-zwart} gives a generalization of the scaling procedure to
interpret various PDMPs as the limit of discrete time Markov chains. In the
real world (Internet), the AIMD mechanism allows a good compromise between the
minimization of network congestion time and the maximization of mean
throughput.

Our aim in this paper is to get quantitative estimates for the convergence to
equilibrium of this general TCP window size process.This process $X$ is
ergodic and admits a unique invariant law, as can be checked using a suitable
Lyapunov function (for instance $V(x)=1+x$, see e.g.
\cite{MR2381160,MR2385873,MR1287609} for the Meyn-Tweedie-Foster-Lyapunov
technique). Nevertheless, this process is irreversible since time reversed
sample paths are not sample paths and it has infinite support. This makes
Meyn-Tweedie-Foster-Lyapunov techniques inefficient for the derivation of
quantitative exponential ergodicity.

The \emph{embedded chain} $\hat X$ of the process $X$ is the sequence of its
positions just after a jump. It is an homogeneous discrete time Markov chain
with state space $[0,\infty)$. As already observed in \cite{DGR}, it is also
the square root of a first order auto-regressive process with non-Gaussian
innovations and random coefficients. We obtain the following results
concerning $\hat X$. We show first that it admits a unique invariant
probability measure $\nu$, and that it converges in law to $\nu$ given any
(random) initial value $\hat X_0$. More precisely, using a coupling technique
on trajectories, we prove an ergodic theorem of geometric convergence to
equilibrium with respect to any Wasserstein distance. Then we provide non
asymptotic concentration bounds, thanks to Gross's logarithmic Sobolev
inequalities.

Similarly, the continuous time process $X$ admits a unique invariant
probability measure $\mu$, and converges in law to $\mu$, for any (random)
initial value $X_0$. The reader may find explicit series for the moments of
$\mu$ and $\nu$ in \cite{GRZ,maulik-zwart,MR2197972}. Nevertheless, quantitative 
convergence to equlibium have not yet been obtained. We will adress this question 
for a slight generalization of the TCP process given by its infinitesimal genrerator:
\begin{equation}\label{eq:Ga}
  L_a(f)(x) = f'(x) + (x+a)\int_0^1(f(h x)-f(x))H(dh)
\end{equation}
where $a\geq 0$. We obtain a good answer if $a>0$. In this case we first show the
existence of a coupling with exponential decay. We use this result to prove an
exponential ergodic theorem in term of Wasserstein distance. Eventually, we
provide a uniform bound over the starting law that implies strong ergodicity.
This kind of uniform estimates, though classical for processes on a compact
set, is rather unusual for real valued processes. Nevertheless, if $a=0$, we 
are not able to derive exponential bounds. 

The remainder of the paper is organized as follows. In the next preliminary
section, we introduce some notations and give the statements of the main
results. In section \ref{se:embedded}, we focus on the embedded chain $\hat X$
and establish its convergence to equilibrium. The last section is devoted to
the study of the continuous time process $X$ and its generalization and contains 
the proof of the results announced in section \ref{se:results}.

\section{Notations and main results}\label{se:results}

Let us first explain how the trajectories of the process $X$ may be
constructed. The \emph{jump rate} (or \emph{jump intensity}) of $X$ is
given by $\lambda(x)=x$ for every $x\in[0,\infty)$. If $X_0=x$ then
the process will experience its first jump at a random time $T$
solution of
$$
\int_0^T\!\lambda(X_s)\,ds=E,
$$
where $E$ is an exponential random variable of unit mean. Since the
trajectories of $X$ are piecewise deterministic with slope 1, this is
nothing else but
$$
\int_0^T\lambda(x+s)\,ds=E,
$$
which leads to $T=\sqrt{x^2+2E}-x$. Then, the sample paths of the process $X$
generated by \eqref{eq:G1} may be constructed recursively as follows. Let
$X_0$ be its non-negative random initial position, ${(E_n)}_{n\geq 1}$ be a
sequence of i.i.d.\ exponential random variables of unit mean, and
${(Q_n)}_{n\geq 1}$ be a sequence of i.i.d.\ random variables of law $H$.
Assume that $X_0$, ${(E_n)}_{n\geq1}$ and ${(Q_n)}_{n\geq1}$ are independent.
We define by induction the jump times ${(T_n)}_{n\geq 1}$ and the positions
just after the jumps ${(X_{T_n})}_{n\geq 1}$ as
\begin{equation}\label{eq:Tn}
  T_n=T_{n-1}+\sqrt{X_{T_{n-1}}^2+2E_n}-X_{T_{n-1}} %
  \quad\text{and}\quad  %
  X_{T_n}=Q_n\sqrt{X_{T_{n-1}}^2+2 E_n}. 
\end{equation}
If we set $T_0=0$, then for every $n\geq 0$ and $t\in [T_n,T_{n+1})$, we have
$X_t=X_{T_n}+t-T_n$ and in particular, $X_{T_{n}}=Q_n X_{T_n^-}$. For every
$t\geq0$, one can also write the series representation
$$
X_t=\sum_{n=0}^\infty (X_{T_n}+t-T_n)\,\mathds{1}_{[T_n,T_{n+1})}(t).
$$
The sequence $\hat X={(X_{T_n})}_{n\geq0}$ is the \emph{embedded chain} of
$X$. According to \eqref{eq:Tn}, this discrete time Markov chain with state
space $[0,\infty)$ satisfies the recursion
\begin{equation}\label{eq:AR}
  \hat X_{n+1}^2=Q_{n+1}^2(\hat X_n^2+2E_{n+1}).
\end{equation}
Thus, the embedded chain $\hat X$ is the square root of a first order
auto-regressive process with non-Gaussian innovations $(2Q_n^2E_n)_{n\geq1}$
and random coefficients ${(Q_n^2)}_{n\geq1}$, as already observed in
\cite{DGR}. The embedded chain $\hat X$ is homogeneous, and its transition
kernel $K$ is given, for any $x\geq0$ and every bounded measurable
$f:[0,\infty)\to\dR$, by the formula
\begin{equation}\label{eq:K}
  K(f)(x) %
  = \int_0^\infty\!f(y)\,K(x,dy) %
  = \dE\SBRA{f\PAR{Q\sqrt{x^2+2E}}}
\end{equation}
where $E$ is an exponential random variable of unit mean and $Q$ is a random
variable of law $H$ independent of $E$. We show in section \ref{se:embedded}
that the embedded Markov chain $\hat X$ admits a unique invariant probability
measure $\nu$, and converges in law to $\nu$ given any (random) initial value
$\hat X_0$. Similarly, the continuous time process $X$ admits a unique
invariant probability measure $\mu$, and converges in law to $\mu$, for any
(random) initial value $X_0$. We recall that explicit series for the moments
of $\mu$ and $\nu$ can be found in \cite{GRZ,maulik-zwart,MR2197972}.

Despite the apparent simplicity of the dynamics \eqref{eq:G1}, the
quantitative study of the long time behavior of $X$ is not easy,
mainly because the jump rate depends on the position $x$ of the
process. Our strategy is to couple two trajectories starting at two
different points in such a way that they get closer and closer. It
seems difficult to stick the two trajectories in order to get total
variation estimates since the sample paths are parallel between jump
times. Thus, we provide quantitative bounds in terms of the
Wasserstein coupling distance. Recall that for every $p\geq1$, the
$W_p$ Wasserstein distance between two laws $\mu$ and $\nu$ on $\dR$
with finite $p^\text{th}$ moment is defined by
\begin{equation}\label{eq:Wp}
  W_p(\mu,\nu) %
  = \left(\inf_{\Pi}\int_{\dR^2}\!|x-y|^p\,\Pi(dx,dy)\right)^{p^{-1}}
\end{equation}
where the infimum runs over all coupling of $\mu$ and $\nu$. In other words,
$\Pi$ runs over the convex set of laws on $\dR^2$ with marginals $\mu$ and
$\nu$, see e.g. \cite{MR1105086,MR1964483}. It is well known that for any
$p\geq1$, the convergence in $W_p$ Wasserstein distance is equivalent to weak
convergence together with convergence of all moments up to order $p$.

The jump part of $L$ ensures that the process will remain essentially in a
compact set. The jumps act in a way like a confining potential. On the other
hand, the jump rate is small when the process is close to the origin. This
prevents the decay of the Wasserstein distance to be exponential for small
times. 

In section \ref{se:embedded} we first establish the following geometric
convergence to equilibrium of the embedded Markov chain $\hat X$ for any
Wasserstein distance.

\begin{thm}[Wasserstein exponential ergodicity for the generic
  embedded chain]\label{th:convK}
  Let $X=(X_t)_{t\geq0}$ and $Y=(Y_t)_{t\geq0}$ be two processes
  generated by \eqref{eq:G1}. Assume that $\cL(X_0)$ and $\cL(Y_0)$
  have finite $p^\text{th}$ moment for some real $p\geq 1$. Let $\hat
  X$ and $\hat Y$ be the embedded chains of $X$ and $Y$. Then, for any
  $n\geq 0$, with a random variable $Q\sim H$,
  $$
  W_p(\cL(\hat X_n),\cL(\hat Y_n))%
  \leq \dE(Q^p)^{n/p} W_p(\cL(X_0),\cL(Y_0)).
  $$
  In particular, if $\nu$ is the invariant law of $\hat X$ then
  $$
  W_p(\cL(\hat X_n),\nu)%
  \leq \dE(Q^p)^{n/p} W_p(\cL(X_0),\nu).
  $$
\end{thm}

We also establish in section \ref{se:embedded} non asymptotic concentration
bounds in the ergodic theorem by using Gross logarithmic Sobolev inequalities:

\begin{thm}[Gaussian deviations for the ergodic theorem for the embedded
  chain]\label{thm:gdevemb}
  Let $\hat X$ be the embedded chain associated to \eqref{eq:G1} and starting
  from $\hat X_0=x\geq0$. Assume that $H$ is the Dirac mass at point
  $\delta\in (0,1)$. Then for any $u\geq 0$ and any 1-Lipschitz function
  $f:[0,\infty)\to\dR$,
  $$
  \dP\PAR{\ABS{\frac{1}{n}\sum_{k=1}^nf(\hat X_k)-\int\!
      f\,d\nu}\geq u+\frac{\delta}{1-\delta}W_1(\delta_x,\nu)}\leq
  2\exp\PAR{-\frac{n(1-\delta^2)u^2}{2\delta^2}}.
  $$
\end{thm}

The convergence to equilibrium of the continuous time process $X$ with generator 
\eqref{eq:Ga} is addressed in section \ref{se:coupling}. The idea is to exhibit a
\emph{coupling}, \emph{i.e.} a Markov process on $[0,\infty)^2$ for which the
marginal components are generated by \eqref{eq:Ga}, with prescribed initial
laws. The infinitesimal generator $\overline{L}$ of this coupling is defined
for every smooth $f:[0,\infty)^2\to\dR$ by
\begin{equation}\label{eq:gen2}
  \overline{L}(f)(x,y)=\mathrm{div}(f)(x,y) %
  +(x+a)\int_0^1\!\PAR{f(h x,h y)\frac{y+a}{x+a}+f(h x,y)\frac{x-y}{x+a}-f(x,y)}\,H(dh)
\end{equation}
if $x\geq y$ and 
$$
\overline{L}(f)(x,y)=\mathrm{div}(f)(x,y) %
+(y+a)\int_0^1\!\PAR{f(h x,h y)\frac{x+a}{y+a}+f(x,h y)\frac{y-x}{y+a}-f(x,y)}\,H(dh)
$$
if $x\leq y$, where $\mathrm{div}(f)=\partial_1 f+\partial_2 f$. This coupling is 
the only one such that the lower component never jump alone. Let us give
 the pathwise interpretation of this coupling. All the heuristic statements below 
 are made more precise hereafter. The positions of both ``components'' increase 
 linearly with slope 1. The jump rate is an increasing function of the position. 
 Thus, ``the higher a component is, the sooner it will jump''. The dynamics 
 of the couple of components is as follows:
\begin{enumerate}
\item After an ``appropriate'' time which depends only on the initial
  position of the upper component, this one jumps.
\item Simultaneously, the other one ``tosses an appropriate coin''
  whose probability of success depends on the positions on the two
  components to decide whether or not it jumps too.
\item In the case of joint jumps, both components use the same
  multiplicative factor.
\item Then, we repeat these three first steps again and again\ldots
\end{enumerate}
This coupling provides the following quantitative exponential upper bounds.
\begin{thm}[Wasserstein exponential ergodicity]\label{th:wun-gen}
 Assume that $a>0$. For any processes $(X_t)_{t\geq0}$ and $(Y_t)_{t\geq0}$
  generated by \eqref{eq:Ga} with finite first moment at initial time,
  and for any $t>0$, we have
  $$
  W_1\PAR{\cL(X_t),\cL(Y_t)} %
  \leq e^{-a\kappa_1 t} W_1(\cL(X_0),\cL(Y_0)),
  $$
  where $\kappa_1=1-\int_0^1\!h\,H(dh)$. 
  In particular, when $Y_0$ follows the invariant law $\mu$ of
  \eqref{eq:Ga}, we get for every $t\geq 0$
  $$
  W_1\PAR{\cL(X_t),\mu} %
  \leq e^{-a\kappa_1 t}\,W_1(\cL(X_0),\mu).
  $$
\end{thm}

The following theorem,
proved in section \ref{se:coupling}, shows that the convergence to equilibrium
is in fact uniform over the starting laws, as it could be for a process living
in a compact set.

\begin{thm}[Strong ergodicity]\label{th:strong}
  Assume that $a>0$. For two processes $X=(X_t)_{t\geq0}$ and
  $Y=(Y_t)_{t\geq0}$ generated by \eqref{eq:Ga} with arbitrary initial
  laws $\cL(X_0)$ and $\cL(Y_0)$ and for every $t$ and $s$ such that
  $t>s>0$, one has
  $$
  W_1(\cL(X_t),\cL(Y_t))%
  \leq \frac{2e^{a\kappa_1 s}}{d\tanh(ds)} e^{-a\kappa_1 t}.
  $$ 
\end{thm}

Theorem \ref{th:strong} provides in particular a uniform bound in
$N\in(0,\infty)$ if $X_0=0$ and $Y_0=N$. This kind of uniform estimates are
classical for processes on a compact set but rather unusual for real valued
ones.

\begin{thm}\label{th:realtcp}
  Assume that $a=0$ and that $H=\delta_h$ with $h\in (0,1)$. Then the process 
  $(X,Y)$ driven by the infinitesimal generator $\overline L$ defined in \eqref{eq:gen2} satisfies
\begin{equation}\label{eq:gronwall}
  \frac{d}{dt}\dE_{(x,y)}\PAR{\ABS{X_t-Y_t}}\leq-(1+h)\dE_{(x,y)}\PAR{\ABS{X_t-Y_t}^2}
\end{equation}
  for any $x,y\in\dR$. In particular, for any $t\geq 0$ and $X_0,Y_0\geq 0$, we have
\begin{equation}\label{eq:unsurt}
  \dE\PAR{\ABS{X_t-Y_t}}\leq \frac{\dE\PAR{\ABS{X_0-Y_0}}}{1+(1+h)\dE\PAR{\ABS{X_0-Y_0}}t}.
 \end{equation}
\end{thm}

\subsubsection*{Open questions and further remarks} 

The inequality \eqref{eq:gronwall} should provide a better bound than \eqref{eq:unsurt}. 
As pointed out in Lemma \ref{le:youpee}, one can actually expect an exponential rate,
but this remains an open problem. One may also ask for a version involving $W_p$ for any $p\geq1$
or even the total variation distance. 

Beyond the TCP window size dynamics, one
may ask about the speed of convergence of ergodic PDMPs, for which necessary
and sufficient ergodicity criteria are already known, see e.g.
\cite{MR2385873}. One may also study the long time behavior of interacting
processes associated to \eqref{eq:G1} or \eqref{eq:G2}, for instance Mac
Kean-Vlasov mean field interactions as in \cite{graham-robert}.

\section{Embedded chain}\label{se:embedded}

It is shown in \cite[Proposition 8]{DGR}, by Laplace transform inversion, that
if $H$ is a Dirac mass at point $\delta\in (0,1)$, the invariant measure of
the embedded chain $\nu=\nu_\delta$ has Lebesgue density
\begin{equation}\label{eq:nudelta}
  x\geq0\mapsto \frac{1}{\prod_{n=1}^\infty(1-\delta^{2n})} %
  \sum_{n=1}^\infty %
  \frac{(-1)^{n-1}\delta^{-2n}}{\prod_{k=1}^{n-1}|1-\delta^{-2k}|} %
  xe^{-\delta^{-2n}x^2/2}.
\end{equation}
It is unimodal, of order $O(x\exp(-\delta^2x^2/2))$ when $x\to\infty$, and all
its derivatives vanish at $x=0$. 

If $H$ is not a Dirac mass, the invariant measure $\nu$ of the embedded Markov
chain is no longer explicit. Nevertheless, the recursion formula \eqref{eq:AR}
provides the following result, see \cite{MR1387886,MR1797309}, which establish
the existence of an invariant measure with sub-Gaussian tails.

\begin{thm}[Convergence of the embedded chain, \cite{MR1387886,MR1797309}]\label{th:cvemb}
  Given any $\hat X_0$, the embedded Markov chain $\hat X=(\hat X_n)_{n\geq0}$
  associated to the dynamics \eqref{eq:G1} converges in distribution to the
  law of the random variable
  $$
  \PAR{2\sum_{n=1}^\infty Q_1^2\cdots Q_n^2E_n}^{1/2}
  $$
  which is a.s. finite, where $E_1,E_2,\ldots$ and $Q_1,Q_2,\ldots$
  are independent sequences of i.i.d.\ random variables following
  respectively the exponential law of unit mean and the law $H$ which
  appear in \eqref{eq:G1}.  In particular, $\nu$ is the unique
  invariant law of $\hat X$ and
  $$
  \int e^{sx^2}\nu(dx)
  =\dE\PAR{\frac{1}{\prod_{n=1}^\infty(1-2sQ_1^2\cdots Q_n^2)}},
  $$
  which is finite if $2sq^2<1$ and infinite if $2sq^2>1$, where
  $q=\inf\BRA{x,\ \dP(Q>x)=1}\leq 1$.
\end{thm}

Let us now turn to our quantitative estimate for the convergence to
equilibrium for the embedded chain.

\begin{proof}[Proof of Theorem \ref{th:convK}]
  It is sufficient to provide a good coupling. Let $x\geq0$ and
  $y\geq0$ be two non-negative real numbers, and let ${(E_n)}_{n\geq
    1}$ and ${(Q_n)}_{n\geq 1}$ be two independent sequences of
  i.i.d.\ random variables with respective laws the exponential law of
  unit mean and the law $H$ which appears in \eqref{eq:G1}. Let $\hat
  X$ and $\hat Y$ be the discrete time Markov chains on $[0,\infty)$
  defined by
  \begin{align*}
    \hat X_0&=x\quad\text{and}\quad %
    \hat X_{n+1}=Q_{n+1}\sqrt{\hat X_n^2+2E_{n+1}}%
    \quad\text{for any $n\geq0$}\\
    \hat Y_0&=y\quad\text{and}\quad %
    \hat Y_{n+1}=Q_{n+1}\sqrt{\hat Y_n^2+2E_{n+1}}%
    \quad\text{for any $n\geq0$}.
  \end{align*}
  By the analogue of \eqref{eq:Tn} for \eqref{eq:G2}, the law of $\hat X$
  (respectively $\hat Y$) is the law of the embedded chain of a process
  generated by \eqref{eq:G1} and starting from $x$ (respectively $y$). Now,
  for any $p\geq 1$, since $x\mapsto\sqrt{x^2+a}$ is a 1-Lipschitz function on
  $[0,\infty)$ for any $a\geq 0$, we get
  \begin{align*}
    \dE\PAR{\ABS{\hat X_{n+1}-\hat Y_{n+1}}^p}&=
    \dE\PAR{Q_{n+1}^p\ABS{\sqrt{\hat X_n^2+2E_{n+1}}-\sqrt{\hat Y_n^2+2E_{n+1}}}^p}\\
    &\leq \dE\PAR{Q_{n+1}^p\ABS{\hat X_n-\hat Y_n}^p}
    =\dE\PAR{Q_{n+1}^p}\dE\PAR{\ABS{\hat X_n-\hat Y_n}^p}
  \end{align*}
  A straightforward recurrence leads to
  $$
  \dE\PAR{\ABS{\hat X_{n}-\hat Y_{n}}^p}%
  \leq \dE\PAR{Q_1^p}^n\ABS{x-y}^p.
  $$
  This gives the desired inequality when the initial laws are Dirac masses.
  The general case follows by integrating this inequality with respect to
  couplings of the initial laws.
\end{proof}

Let us now investigate some properties of the kernel $K$ defined by
\eqref{eq:K} that will be used to provide concentration bounds for the ergodic
theorem. The key point is that $K^n$ and $\nu$ satisfy a Gross (or logarithmic
Sobolev) inequality.

\begin{defi}[Gross inequality]
  A law $\eta$ on $\dR^d$ satisfies a Gross (or logarithmic Sobolev
  \cite{MR1845806,MR0420249}) inequality with constant $c>0$ when for any
  smooth compactly supported $f:\dR^d\to\dR$,
  $$
  \int\!f^2\log(f^2)\,d\eta-\int\!f^2\,d\eta\log\!\int\!f^2\,d\eta %
  \leq c\,\!\int\!|\nabla f|^2\,d\eta.
  $$
  We denote by $\textsc{Gross}(\eta)\in(0,\infty]$ the smallest
  constant for which this holds true.
\end{defi}

If $F\cdot\eta$ is the image of $\eta$ by $F$ then
$\textsc{Gross}(F\cdot\eta)\leq\textsc{Gross}(\eta)\NRM{F}^2_\textsc{Lip}$.
The Gross inequality contains an information on Gaussian concentration
of measure: the function $x\mapsto e^{ax^2}$ is $\eta$-integrable as
soon as
$a<1/\textsc{Gross}(\eta)$. 
Moreover, if $\eta$ has covariance $\Sigma$ with spectral radius
$\rho(\Sigma)$ then $2\rho(\Sigma)\leq\textsc{Gross}(\eta)$ and
equality is achieved when $\eta$ is Gaussian. Furthermore, for any
$\alpha$-Lipschitz function $f:\dR\to\dR$ and any $\lambda>0$,
\begin{equation}\label{eq:lapgross}
  \dE_\eta\PAR{e^{\lambda f}}\leq e^{C\alpha^2\lambda^2/4} e^{\lambda\dE_\eta f}
\end{equation}
as soon as $C\geq\textsc{Gross}(\eta)$. This means that $\eta$ satisfies a
\emph{sub-Gaussian concentration of measure} for Lipschitz functions. For more
details, see e.g. \cite{MR1849347,MR1964483} and references therein.

\begin{thm}[Properties of the kernel of the embedded chain]\label{th:grossK}
  Let $\hat X$ be the embedded chain associated to \eqref{eq:G1} with
  transition kernel \eqref{eq:K}. Assume that $H$ is the Dirac mass at point
  $\delta\in[0,1)$. If $f$ is a 1-Lipschitz function from $[0,+\infty)$ to
  $\dR$, then $x\mapsto K(f)(x)$ is a $\delta$-Lipschitz function from
  $[0,+\infty)$ to $\dR$. Moreover, for any $x\geq 0$, the law $K(\cdot)(x)$
  satisfies a Gross inequality with constant $2\delta^2$.
\end{thm}

\begin{proof}
  If $\delta=0$, then $K$ is the Dirac mass at 0 and the result is
  trivial. For any smooth function $f:[0,\infty)\to\dR$, we have from
  \eqref{eq:K} that
  \begin{equation}\label{eq:comK}
    \ABS{(Kf)'}%
    =\delta\ABS{K\PAR{\frac{x}{\sqrt{x^2+2
            E}}f'\PAR{\delta\sqrt{x^2+2E}}}}%
    \leq \delta K(\ABS{f'}).
  \end{equation}
  Let us show now that for every $x\geq 0$ the law
  $K(x,\cdot)=\cL(\hat X_1\,\vert\,\hat X_0=x)$ satisfies a Gross
  inequality with constant $2\delta^2 $. Since $E$ is exponential of
  mean $1$, the law $\eta$ of $\sqrt{E/2}$ is a $\chi$-distribution
  with probability density and cumulative distribution functions given
  by
  $$
  g:v\mapsto 4ve^{-2v^2}\mathds{1}_\BRA{v>0} \quad\text{and}\quad %
  G:v\mapsto (1-e^{-2v^2})\mathds{1}_\BRA{v>0}.
  $$
  On the other hand, $2E\overset{d}{=}U_1^2+U_2^2$ where $U_1,U_2$ are i.i.d.\
  standard Gaussians, and thus
  $$
  \sqrt{E/2}%
  =\frac{1}{2}\sqrt{2E}%
  \overset{d}{=}\frac{1}{2}\sqrt{U_1^2+U_2^2}.
  $$
  Also, $\eta$ is the image of the Gaussian law $\cN(0,I_2)$ on $\dR^2$ by a
  $(1/2)$-Lipschitz function, and this implies that $\eta$ satisfies a Gross
  inequality with constant $1/2$. Moreover,
  \begin{align*}
  K(f)(x)%
  &=\int_0^\infty\!f\PAR{\delta\sqrt{x^2+2u}}e^{-u}\,du\\
  &=\int f(2\delta v) \frac{4ve^{-2v^2}}{e^{-x^2/2}}\mathds{1}_\BRA{v>x/2}\,dv\\
  &=\int f(2\delta v) \frac{g(v)}{1-G(x/2)}\mathds{1}_\BRA{v>x/2}\,dv.
  \end{align*}
  Thus, $K(\cdot)(x)$ is just the image law by the Lipschitz map $v\mapsto
  2\delta v$ of the law $\eta$ conditioned on $(x/2,+\infty)$. This
  conditional law is in turn the image of $\eta$ by the function
  $$
  t\mapsto G^{-1}(G(x)+(1-G(x))G(t))=G^{-1}(1-\exp(-t^2-x^2))
  =\sqrt{x^2+t^2}.
  $$
  This function is 1-Lipschitz for any $x\geq 0$. Consequently, by using twice
  the stability of Gross inequalities by Lipschitz maps, we obtain that for
  every $x>0$, the law $K(x,\cdot)$ satisfies a Gross inequality with constant
  $(2\delta)^2/2=2\delta^2$.
\end{proof}

\begin{rem}
  When $\delta=0$, the embedded chain is the constant Markov chain equal to 0.
  Moreover, the chain ${(Z_n)}_{n\geq 0}$ defined by $Z_n=X_{T_n^-}$ is also
  quite simple to study. Indeed, the random variables ${(Z_n)}_{n\geq 1}$ are
  i.i.d.\ and have the law $\nu$ of $\sqrt{2E}$. The previous proof ensures
  that $\nu$ satisfies a Gross inequality with constant 2. One of the most
  useful properties of Gross inequality is the tensorization property:
  $\mathrm{Gross}(\eta^{\otimes n})\leq \mathrm{Gross}(\eta)$ for every
  $n\geq1$, see e.g. \cite[Chapter 1]{MR1845806}. Using now the concentration
  property, one has, for any 1-Lipschitz function and any $u\geq 0$,
  $$
  \dP\PAR{\ABS{\frac{1}{n}\sum_{k=1}^nf(Z_k)-\int\!f\,d\nu}\geq u}%
  \leq 2\exp\PAR{-\frac{Nu^2}{2}}.
  $$
\end{rem}

In the more general case where $\delta$ is positive, ${(\hat X_n)}_{n\geq 1}$
is no longer i.i.d. Nevertheless, the Gross inequality holds true for the
iterated kernels and for the invariant law $\nu$:

\begin{cor}[Gross inequality for the embedded chain and its invariant
  law $\nu$]\label{cor:lsiemb}
  Let $\hat X$ be the embedded chain associated to \eqref{eq:G1}. Assume that
  $H$ is the Dirac mass at point $\delta\in(0,1)$. For every $n\geq0$, let
  $K^n$ be the iterated transition kernel of $\hat X$, defined recursively for
  every bounded measurable function $f:[0,\infty)\to\dR$ by
  $$
  K^0(f)=f %
  \quad\text{and}\quad %
  K^{n+1}(f)=K(K^n(f))
  $$
  where $K$ is the kernel of $\hat X$ as in \eqref{eq:K}. Then for every
  integer $n\geq1$ and every real $x\geq0$, the iterated kernel $K^n(x,\cdot)$
  of $\hat X$ satisfies a Gross inequality and
  $$
  \textsc{Gross}(K^n(x,\cdot))\leq 2\delta^2\frac{1-\delta^{2n}}{1-\delta^2}.
  $$
  Also, the invariant law $\nu$ of $\hat X$ (see theorem \ref{th:cvemb})
  satisfies a Gross inequality and
  $$
  \textsc{Gross}(\nu)\leq 2\delta^2(1-\delta^2)^{-1}.
  $$
\end{cor}

\begin{proof}
  Recall that for every $n\geq0$, $x\geq0$, and bounded measurable
  $f:[0,\infty)\to\dR$,
  $$
  \dE \PAR{f(\hat X_n)\,|\,\hat X_0=x} %
  =(K^nf)(x) %
  = \int_0^\infty\!f(y)\,K^n(x,dy) 
  $$
  To show that $K^n$ satisfies a Gross inequality, we use a semi-group
  decomposition technique borrowed from \cite{MR2208718}. For any $n\geq1$ and
  any smooth function $f:[0,\infty)\to\dR$, the quantity
  $$
  E_n(f):=K^n(f^2\log f^2) - K^n(f^2) \log K^n(f^2)
  $$
  is equal to the telescopic sum
  $$
  \sum_{i=1}^n \BRA{K^i\SBRA{K^{n-i}(f^2)\log K^{n-i}(f^2)}%
    - K^{i-1}\SBRA{K^{n-i+1}(f^2)\log K^{n-i+1}(f^2)}}.
  $$
  Since the measure $K(\cdot)(x)$ satisfies a Gross inequality of constant
  $2\delta^2$, we get, with $g_{n-i}=\sqrt{K^{n-i}(f^2)}$,
  $$
  E_n(f)=\sum_{i=1}^n K^{i-1} \SBRA{E_1(g_{n-i})}  %
  \leq 2\delta^2\sum_{i=1}^nK^i\PAR{|\nabla g_{n-i}|^2},
  $$
  Now, by using the commutation \eqref{eq:comK}, we obtain, for all $1\leq i
  \leq n$,
  $$
  |\nabla g_{n-i}|^2 %
  = \frac{ \ABS{\nabla K^{n-i}(f^2)}^2}{4 K^{n-i}(f^2)}%
  \leq \delta^2\frac{\PAR{K\ABS{\nabla K^{n-i-1}(f^2)}}^2}{4KK^{n-i-1}(f^2)}
  $$
  Next, the Cauchy--Schwarz inequality
  $$
  \frac{(Kf)^2}{K(g)} \leq K\PAR{\frac{f^2}{g}}
  $$
  gives
  $$
  \frac{\PAR{K\ABS{\nabla K^{n-i-1}(f^2)}}^2}{4KK^{n-i-1}(f^2)}
  \leq
  K\PAR{\frac{\ABS{\nabla K^{n-i-1}(f^2)}^2}{4K^{n-i-1}(f^2)}}
  =K\PAR{\ABS{\nabla g_{n-i-1}}^2}. 
  $$
  From these bounds, a straightforward induction gives
  $$
  \ABS{\nabla g_{n-i}}^2%
  \leq \delta^{2(n-i)}K^{n-i}\PAR{\ABS{\nabla f}^2}.
  $$
  Consequently, by putting all together, we have 
  \begin{align*}
    E_n(f^2)%
    \leq 2\delta^2\sum_{i=0}^{n-1} \delta^{2i} K^n\PAR{\ABS{\nabla f}^2} %
    = 2\delta^2\frac{1-\delta^{2n}}{1-\delta^2} K^n\PAR{\ABS{\nabla f}^2}.
  \end{align*} 
  This gives 
  $$
  \textsc{Gross}(K^n)\leq 2\delta^2(1-\delta^{2n})(1-\delta^2)^{-1}.
  $$
  Finally, from Theorem \ref{th:cvemb}, $K^n$ tends weakly to $\nu$ as $n$
  tends to infinity and thus
  $$
  \textsc{Gross}(\nu)\leq
  \limsup_{n\to\infty}\textsc{Gross}(K^n)\leq 2\delta^2(1-\delta^2)^{-1}.
  $$
\end{proof}

The Gross inequality for $K$ can also be used to derive 
Theorem \ref{thm:gdevemb}.

\begin{proof}[Proof of Theorem \ref{thm:gdevemb}]
  We shall establish that for any $u\geq 0$ and any 1-Lipschitz function
  $f:[0,\infty)\to\dR$,
  $$
  \dP\PAR{\frac{1}{n}\sum_{k=1}^nf(\hat X_k)-\int\!
      f\,d\nu\geq u+\frac{\delta}{1-\delta}W_1(\delta_x,\nu)}\leq
  \exp\PAR{-\frac{n(1-\delta^2)u^2}{2\delta^2}}
  $$
  and the desired result follows immediately from this bound used for $f$ and
  $-f$. For any 1-Lipschitz function $f$, any $r>0$ and $\lambda>0$, we have,
  $$
  \dP\PAR{\frac{1}{n}\sum_{k=1}^n f(\hat X_k)\geq r} \leq
  \dE\PAR{e^{\lambda\sum_{k=1}^nf(\hat X_k)}} e^{-nr\lambda}.
  $$
  Now the Markov property ensures that
  \begin{align*}
    \dE\PAR{e^{\lambda\sum_{k=1}^nf(\hat X_k)}}
    &=\dE\PAR{e^{\lambda\sum_{k=1}^{n-1}f(\hat X_k)}\dE\PAR{e^{\lambda f(\hat X_n)}|X_{n-1}}}\\
    &=\dE\PAR{e^{\lambda\sum_{k=1}^{n-1}f(\hat X_k)}K\PAR{e^{\lambda
          f}}(X_{n-1})}.
  \end{align*}
  From Theorem \ref{th:grossK}, the kernel $K(x,\cdot)$ of $\hat X$ satisfies
  a Gross inequality with constant $2\delta^2$ for every $x\geq0$. This
  inequality implies by \eqref{eq:lapgross} that for any $c$-Lipschitz
  function $g$,
  $$
  K\PAR{e^{\lambda g}}\leq \exp\PAR{\lambda Kg
    +\frac{c^2\delta^2\lambda^2}{2}}.
  $$
  Consequently, the Laplace transform of the ergodic mean can be
  bounded as follows:
  $$
  \dE\PAR{e^{\lambda\sum_{k=1}^nf(\hat X_k)}} \leq
  \exp\PAR{\frac{\delta^2\lambda^2}{2}}
  \dE\PAR{e^{\lambda\sum_{k=1}^{n-2}f(\hat X_k)} \dE\PAR{e^{\lambda
        (f+Kf)(\hat X_{n-1})}|\hat X_{n-2}}}.
  $$
  The commutation relation \eqref{eq:comK} ensures that $f+Kf$ is
  $(1+\delta)$-Lipschitz and then
  $$
  \dE\PAR{e^{\lambda (f+Kf)(\hat X_{n-1})}|\hat X_{n-2}}%
  \leq \exp\PAR{\frac{(1+\delta)^2\delta^2\lambda^2}{2}}%
  e^{\lambda(Kf+K^2f)(\hat X_{n-2})}.
  $$
  A straightforward recurrence ensures that
  $$
  \dE\PAR{e^{\lambda\sum_{k=1}^nf(\hat X_k)}} \leq
  \exp\PAR{\frac{n\delta^2\lambda^2}{2(1-\delta^2)}}e^{\sum_{k=1}^n
    K^kf(x)}.
  $$
  Choosing $r=(1/n)\sum_{k=1}^n K^kf(x)+u$ leads to
  $$
  \dP\PAR{\frac{1}{n}\sum_{k=1}^nf(\hat
    X_k)-\frac{1}{n}\sum_{k=1}^nK^kf\geq u} \leq
  \exp\PAR{\frac{n\delta^2\lambda^2}{2(1-\delta^2)}-n\lambda u}.
  $$
  The right hand side is minimum for $\lambda=u(\delta^{-2}-1)$. At this
  point, we recall the dual formulation of $W_1(\alpha,\beta)$ for every
  probability laws $\alpha$ and $\beta$:
  $$
  W_1(\alpha,\beta)=\sup_{\NRM{f}_{\text{Lip}}\leq 1}%
  \PAR{\int\! f\,d\alpha-\int\! f\,d\beta}%
  \quad\text{where}\quad%
  \NRM{f}_{\text{Lip}}=%
  \sup_{x\neq y}\frac{\ABS{f(x)-f(y)}}{\ABS{x-y}}.
  $$
  Therefore, by using Theorem \ref{th:convK}, one gets
  $$
  \frac{1}{n}\sum_{k=1}^nK^kf(x)-\int\! f\,d\nu \leq
  \frac{1}{n}\sum_{k=1}^nW_1(K^k(\cdot)(x),\nu) \leq
  \frac{\delta}{1-\delta}W_1(\delta_x,\nu).
  $$
\end{proof}

\begin{rem}
  A careful reading of the proof of Theorem \ref{thm:gdevemb} suggests that
  one may replace the initial law $\delta_x$ by a more general initial law
  provided that it satisfies a sub-Gaussian concentration for Lipschitz
  functions \eqref{eq:lapgross}.
\end{rem}

\section{Continuous time process}\label{se:coupling}

As an introduction of our coupling method to prove Theorem \ref{th:wun-gen},
let us consider the following simpler dynamics, studied recently in
\cite{LL,MR2426601}. The window size is modeled by a Markov process
$X=(X_t)_{t\geq0}$ that increases linearly with rate one. Congestion signals
arrive according to a Poisson process with constant rate $\lambda>0$, and upon
receipt of the $k^\text{th}$ signal, the window size is reduced by
multiplication with a random variable $Q_k$. We assume that ${(Q_k)}_{k\geq0}$
is a sequence of i.i.d.\ random variables of law $H$ with support in $[0,1)$.
In other words, the process $X$ is generated by
\begin{equation}\label{eq:G2}
  L(f)(x) = f'(x) + \lambda\!\!\int_0^1\!(f(h x)-f(x))\,H(dh)
\end{equation}
where $\lambda$ is this time a positive real number. In this framework, one
can compute explicitly the transient moments of $X_t$ (see
\cite{LL,MR2426601}): for every $n\geq0$, every $x\geq0$, and every $t\geq0$,
\begin{equation}\label{eq:momlop}
  \dE((X_t)^n\,\vert\,X_0=x)=\frac{n!}{\prod_{k=1}^n\theta_k}+%
  n!\sum_{m=1}^n\biggr(\sum_{k=0}^m\frac{x^k}{k!}%
  \prod_{\substack{j=k\\j\neq m}}^n\frac{1}{\theta_j-\theta_m}\biggr)e^{-\theta_mt}
\end{equation}
where for every real or integer $p\geq1$ the quantity $\theta_p$ is as in our
Theorem \ref{th:constant}. In contrast with the original dynamics
\eqref{eq:G1}, the jump rate is constant and thus the jumps occur at
Poissonian times. In this framework, we derive easily the following theorem,
which states an exponential ergodicity in all Wasserstein distances.

\begin{thm}[Wasserstein Exponential Ergodicity for constant jump rate]
   \label{th:constant}
   Let $X=(X_t)_{t\geq0}$ and $Y=(Y_t)_{t\geq0}$ be two processes generated by
   \eqref{eq:G2}. Assume that $\cL(X_0)$ and $\cL(Y_0)$ have finite
   $p^\text{th}$ moment for some real $p\geq1$. If one defines
   $\theta_p=\lambda(1-\dE(Q^p))$ with $Q\sim H$ then for every $t\geq0$,
   $$
   W_p(\cL(X_t),\cL(Y_t))\leq
   W_p(\cL(X_0),\cL(Y_0))\,e^{-p^{-1}\theta_pt}.
   $$   
\end{thm}

We ignore if the exponential rate of convergence in Theorem \ref{th:constant}
is optimal. One may try to get an upper bound from the moments formula
\eqref{eq:momlop}.

\begin{proof}[Proof of Theorem \ref{th:constant}]
  Let $N=(N_t)_{t\geq 0}$ be a Poisson process with constant intensity
  $\lambda$ and $Q=(Q_k)_{k\geq 1}$ be i.i.d.\ random variables with
  law $H$, independent of $N$. For any $x,y \geq0$, let us consider
  the processes $X=(X_t)_{t\geq0}$ and $Y=(Y_t)_{t\geq0}$ starting
  respectively at $x$ and $y$ at time 0, that jump when $N$ does, with
  a multiplicative factor $Q_k$ for the $k^{th}$ jump, and increase
  linearly with slope one between these jumps. It is quite clear that
  these processes are generated by \eqref{eq:G2}. Moreover, between
  jumps, $\ABS{X_t-Y_t}$ remains constant and at the $k^{th}$ jump
  this quantity is multiplied by $Q_k$. Thus for every $t\geq 0$ and
  $p\geq0$,
  \begin{align*}
    \dE\PAR{\ABS{X_t-Y_t}^p}&=%
    \sum_{k=0}^\infty\dE\PAR{\ABS{X_t-Y_t}^p\mathds{1}_{\{N_t=k\}}}\\
    &=\ABS{x-y}^p\sum_{k=0}^\infty\dE(Q^p)^k\dP\PAR{N_t=k}\\
    &=\ABS{x-y}^pe^{-\lambda t (1-\dE(Q^p))}.
  \end{align*}
  As a consequence, if $X=(X_t)_{t\geq0}$ and $Y=(Y_t)_{t\geq0}$ are
  now two processes generated by \eqref{eq:G2} with a constant jump
  intensity $\lambda$ and arbitrary initial laws, we obtain that, for
  any coupling $\Pi$ of their initial law $\cL(X_0)$ and $\cL(Y_0)$,
  any $t\geq0$, and any $p\geq1$,
  $$
  W_p(\cL(X_t),\cL(Y_t))^p %
  \leq e^{-\theta_p t}\int_{[0,\infty)^2}\!\ABS{x-y}^p\,\Pi(d(x,y)).
  $$
  Taking the infimum over $\Pi$ concludes the proof.
\end{proof}

Let us now turn to the \emph{generalized} TCP window size process generated by the infinitesimal generator
\eqref{eq:Ga}. Consider a two dimensional process where both components are generated by
\eqref{eq:Ga}. Since the sample paths of both components have slope $1$
between jumps, the distance between them remains constant except at jump
times. If the components jump together with the same factor $Q$, then this
distance is also multiplied by $Q$. Thus, our strategy is to encourage simultaneous 
jumps: let us introduce the Markov process ${((X_t,Y_t))}_{t\geq 0}$ on $[0,\infty)^2$ defined by its infinitesimal generator
\begin{align*}
\overline Lf(x,y)=&\partial_1f(x,y)+\partial_2f(x,y)\\%
&+(x-y)\int_0^1\!(f(hx,y)-f(x,y))\,H(dh)\\
&+(y+a)\int_0^1\!(f(hx,hy)-f(x,y))\,H(dh)
\end{align*}
if $x\geq y$ (if $y<x$ one has to exchange the variables $x$ and $y$).
 
Choosing a test function $f$ of the form $f(x,y)=g(x)$ or $f(x,y)=g(y)$ shows that $X$ and $Y$ are both Markov processes with infinitesimal generator $L$. 

The dynamics of $(X,Y)$ is as follows: if $(X_0,Y_0)=(x,y)$ with for example $x\geq y$, then  
\begin{itemize}
\item the first jump time $T$ has density $t\mapsto (x+t)e^{-t^2/2-xt}\ind_{(0,+\infty)}(t)$, 
\item on the event $\BRA{T=t}$ we have $(X_s,Y_s)=(x+s,y+s)$ for $s<t$ and 
$$
(X_t,Y_t)=
\begin{cases}
\displaystyle{\PAR{\frac{x+t}{2},\frac{y+t}{2}}}& \text{with probability } \displaystyle{\frac{y+t+a}{x+t+a}},\\ 
&\\
\displaystyle{\PAR{\frac{x+t}{2},y+t}}& \text{with probability }  \displaystyle{\frac{x-y}{x+t+a}}. 
\end{cases}
$$
\end{itemize}

\subsection{The modified TCP process}

The first part of this section is dedicated to the proof of Theorem \ref{th:wun-gen}.

\begin{proof}[Proof of Theorem \ref{th:wun-gen}]
We have to study the function $\alpha$: $t\mapsto \dE_{(x,y)}\PAR{\ABS{X_t-Y_t}}$ where $(X,Y)$ evolves according to the generator $\overline L$. Assume that $x>y$, then 
\begin{align*}
\alpha_{(x,y)}'(0)=&(x-y)\int_0^1\!(\ABS{hx-y}-\ABS{x-y})\,H(dh)%
+(y+a)(x-y)\int_0^1\!(h-1)\,H(dh)\\
=&-(x-y)\int_0^1\!\ind_\BRA{hx>y}(1-h)(x+y+a)\,H(dh)\\
&-(x-y)\int_0^1\!\ind_\BRA{hx\leq y}\PAR{(1+h)(x-y)+(1-h)a}\,H(dh) \\
\leq & -a(x-y)\int_0^1\!(1-h)\,H(dh).
\end{align*}
The Markov property ensures that 
$$
\alpha_{(x,y)}'(t)\leq a\kappa_1\alpha_{(x,y)}(t),
$$
where $\kappa_1=1-\int_0^1\!h\,H(dh)$. This obviously implies that 
$$
\dE_{(x,y)}\PAR{\ABS{X_t-Y_t}}\leq \ABS{x-y}e^{-a\kappa_1t}.
$$
The end of the proof is straightforward. 
If $X=(X_t)_{t\geq0}$ and $Y=(Y_t)_{t\geq0}$ are two processes generated by
\eqref{eq:G1} and if $\Pi$ is a coupling of $\cL(X_0)$ and $\cL(Y_0)$, we
hava,for every $t\geq0$,
\begin{align*}
  W_1(\cL(X_t),\cL(Y_t))
  &\leq \int_{[0,\infty)^2}\!\dE\PAR{\ABS{X_t-Y_t}\,|\,X_0=x,Y_0=y}\,\Pi(dx,dy) \\
  &\leq  e^{-a\kappa_1 t}\int_{[0,\infty)^2}\!\ABS{x-y} \,\Pi(dx,dy).
\end{align*}
Taking the infimum over $\Pi$ provides the result.  
\end{proof}

Let us now turn to the proof of Theorem \ref{th:strong}.

\begin{proof}[Proof of theorem \ref{th:strong}]
  The function $f$ defined by $f(x)=x$ for every $x\geq0$ satisfies to
  $$
  Lf(x)=1-\kappa_1x(x+a) \leq 1-\kappa_1x^2
  $$
  where $\kappa_1=1-\int_0^1\!h\,H(dh)\in (0,1]$. 
  Now, for every $x\geq0$ and $t\geq0$,
  \begin{align*}
  \alpha_x(t) %
  &:=\dE(X_t\,\vert\,X_0=x) \\
  &=\alpha_x(0)+\int_0^t\!\alpha_x'(s)\,ds \\%
  &= x+\int_0^t\!\dE((Lf)(X_s)|X_0=x)\,ds \\%
  &\leq x+\int_0^t (1-\kappa_1\dE(X_s^2|X_0=x))\,ds.
  \end{align*}
  Also, since $-\kappa_1$ is negative, we obtain, by using Jensen's inequality,
  $$
  \alpha_x'(t) %
  =1- \kappa_1\dE(X_t^2|X_0=x)\leq 1-\kappa_1\alpha_x(t)^2.
  $$
  As a consequence, $\alpha_x\leq \beta_x$ where $\beta_x$ is the solution of
  the Riccati differential equation
  $$
  \begin{cases}
    \beta_x(0)=x,\\
    \beta'_x(t)=1-\kappa_1\beta_x(t)^2\text{ for $t>0$}
  \end{cases}
  $$
  Denoting $d=\sqrt{\kappa_1}$, one gets, for $x\geq 1/d$,   
  $$
  \beta_x(t)
  =\frac{1}{d}+\frac{2(x-1/d)e^{-2dt}}{(dx+1)-(dx-1)e^{-2dt}}%
  =\frac{1}{d}\frac{dx \cosh(dt)+\sinh(dt)}{dx\sinh(dt)+\cosh(dt)}%
  \leq \frac{1}{d\tanh(dt)},
   $$
  and therefore
  $$
  \sup_{x\geq 1/d}\alpha_x(t)\leq \frac{1}{d\tanh(dt)}.
  $$
  On the other hand, we have also $\sup_{x\leq 1/d}\alpha_x(t)\leq
  1/d$, and thus for every $t>0$,
  $$
  \sup_{x\geq0}\alpha_x(t)\leq
  \frac{1}{d\tanh(dt)}.
  $$
  Consider now two processes $(X_t)_{t\geq0}$ and $(Y_t)_{t\geq0}$
  generated by \eqref{eq:G1} with arbitrary initial laws. For any
  $s>0$, $\dE(|X_s-Y_s|)\leq 2\sup_x\alpha_x(s)$ and therefore the
  upper bound above gives
  $$
  W_1(\cL(X_s),\cL(Y_s)) \leq \frac{2}{d\tanh(ds)}.
  $$
  Together with Theorem \ref{th:wun-gen}, this gives the following uniform
  estimate, for every $t\geq s>0$,
  \begin{align*}
    W_1(\cL(X_t),\cL(Y_t))%
    &\leq  W_1(\cL(X_s),\cL(Y_s))e^{-a\kappa_1(t-s)}\\%
    &\leq \frac{2e^{a\kappa_1 s}}{d\tanh(ds)} e^{-a\kappa_1 t}.
  \end{align*}
\end{proof}

\subsection{The real TCP process}

We end by giving the proof of Theorem \ref{th:realtcp} and making some comments on it.

\begin{proof}[Proof of Theorem \ref{th:realtcp}]
 We start the proof as in Theorem \ref{th:wun-gen}:
 $$
\alpha_{(x,y)}'(0)=
\begin{cases}
-(1-h)(x+y)(x-y)&\text{if }hx>y,\\
-(1+h)(x-y)^2&\text{if }hx\leq y.
\end{cases}
$$
The first bound is better. Nevertheless, if $D_h$ is the set $\BRA{(x,y),\ hy\leq x\leq h^{-1}y}$, one has to notice that the process $(X,Y)$ cannot exit $D_h$. Then, thanks to Markov property, we get the following bound: 
$$
\frac{d}{dt}\dE\PAR{\ABS{X_t-Y_t}}\leq-(1+h)\dE\PAR{\ABS{X_t-Y_t}^2}.
$$
Jensen's inequality ensures that 
$$
\frac{d}{dt}\dE\PAR{\ABS{X_t-Y_t}}\leq-(1+h)\BRA{\dE\PAR{\ABS{X_t-Y_t}}}^2,
$$
and thus, for any $t\geq 0$,  
$$
\dE\PAR{\ABS{X_t-Y_t}}\leq \frac{\dE\PAR{\ABS{X_0-Y_0}}}{1+(1+h)\dE\PAR{\ABS{X_0-Y_0}}t}.
$$
\end{proof}

Figure \ref{fi:comp} suggests that the convergence rate given by Theorem \ref{th:realtcp} is far from being satisfactory. The non-optimality of the coupling is clear. However, even with such a coupling, one could expect an explicit exponential upper bound. Let us denote $D_t=\ABS{X_t-Y_t}$ where ${(X_t,Y_t)}$ is defined in Theorem \ref{th:realtcp} . We think that $\dE(D_t^2)$ is in fact of the order of $\dE(D_t)$ (instead of $\dE(D_t)^2$). Indeed, with a rate of order $D_t$ a nonsimultaneous jump occurs at time $t$ and then $D_t$ is again of order one. In the following lemma, we introduce a simple Markov chain which captures the essential feature the dynamics of $D_t$ (division by 2 with probability $1-O(D_t)$) and we show that the expected position at time $n$ goes to zeros exponentially fast as $n$ goes to infinity. Additionally the recursive equation \eqref{eq:formomtoy} plays the role of \eqref{eq:gronwall}.


\begin{figure}[htbp]%
  \begin{center}%
  \includegraphics[scale=0.6]{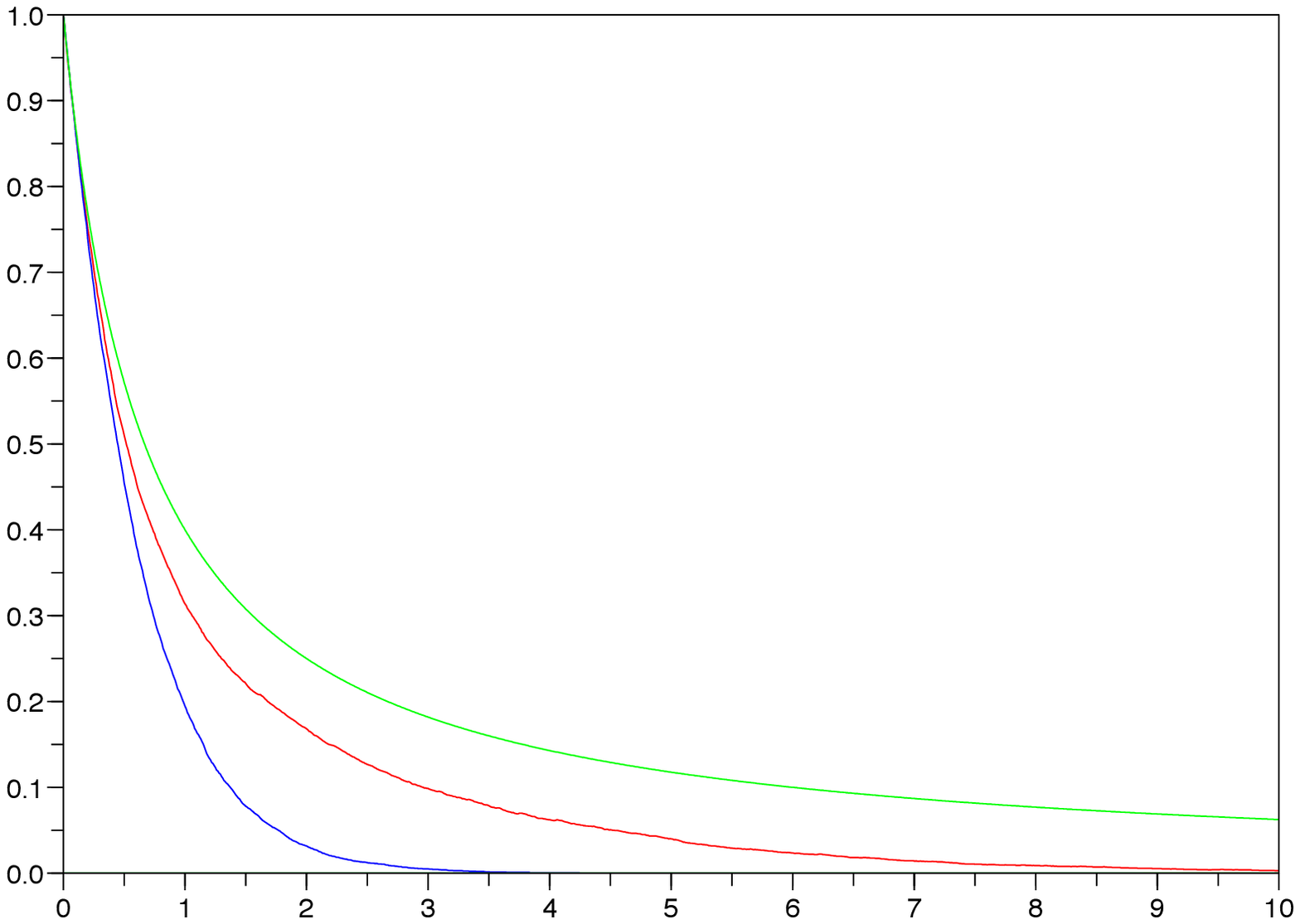}%
  \caption{{Here $(x,y)=(2,1)$ and $H=\delta_{1/2}$. This picture presents the three following functions of time: $t\mapsto W_1(\cL(X^x_t),\cL(X^y_t))$ where $X^u$ is driven by \eqref{eq:G1} with $X^u_0=u$ (blue curve),
$t\mapsto \dE\PAR{\ABS{X^{x,y}_t-Y^{x,y}_t}}$ where $(X^{x,y},Y^{x,y})$ is driven by \eqref{eq:gen2}  with $(X^{x,y}_0,Y^{x,y}_0)=(x,y)$ (red curve), $t\mapsto (x-y)/(1+1.5(x-y)t)$ the upper bound \eqref{eq:unsurt} of Theorem \ref{th:realtcp} (green curve). The first and second curves have been obtained by Monte-Carlo simulations. 
}}%
  \label{fi:comp}%
  \end{center}%
 \end{figure}

\begin{lem}\label{le:youpee}
Consider the homogeneous irreducible Markov chain $X={(X_n)}_{n\geq 0}$ with state space $E=\BRA{2^{-i},\ i\in\dN}$ such that, for any $n\geq 0$ and $x\in E$, on the event $\BRA{X_n=x}$
$$
X_{n+1}=
\begin{cases}
 1&\text{with probability }x/2,\\
  x/2&\text{with probability }1-x/2.
\end{cases}
$$
Denote by $\dE^1(X_n)$ the quantity $\dE(X_n\vert X_0=1)$. Then, for any $n\geq1$,
\begin{equation}\label{eq:formomtoy}
\dE^1\PAR{X_{n+1}}=\dE^1\PAR{X_{n}}-\frac{1}{4}\dE^1\PAR{X_{n}^2}
\end{equation}
and there exists a constant $c>0$ such that for any $n\geq 1$,
\begin{equation}\label{eq:expdectoy}
\dE^1(X_n)\leq \exp\PAR{-cn}.
\end{equation}
\end{lem}

\begin{proof}The Markov chain $X$ is transient (and converges to 0) since 
$$
p:=\dP(\forall n\geq 0,\ X_{n+1}= X_n/2\vert X_0=1)=\prod_{i=1}^\infty (1-2^{-i})>0.
$$
Since for any $n\geq 0$, 
$$
 \dE(X_{n+1}\vert X_{n})=\frac{1}{2} X_{n}+\frac{1}{2} X_{n} \PAR{1-\frac{1}{2} X_{n}},
$$
we get \eqref{eq:formomtoy}.
In particular, $\dE^1(X_1)=3/4$ and $n\mapsto \dE^1(X_n)$ is decreasing. Similarly, 
$$
\dE^1(X_n^2 )=\frac{1}{2}\dE^1(X_{n-1})+\frac{1}{4}\dE^1\PAR{X_{n-1}^2(1-\frac{1}{2}X_{n-1})}%
\geq \frac{1}{2} \dE^1(X_{n-1}) \geq \frac{1}{2} \dE^1(X_{n}).
$$ 
As a consequence, for any $n\geq 1$,  
 $$
 \dE^1\PAR{X_{n+1}}\leq \frac{7}{8}\dE^1\PAR{X_{n}}.
 $$
 and \eqref{eq:expdectoy} follows since
 for any $n\geq 1$,  
$$
\dE^1(X_n)\leq \frac{6}{7}\PAR{\frac{7}{8}}^n.
$$
\end{proof}

\bigskip

{\small \textbf{Acknowledgments.} The authors are grateful to the anonymous
  referees for their quick reports and fine comments which helped to improve
  the paper. DC and FM would like to thank Philippe Robert and his team for kind
  hospitality during their visits to INRIA Rocquencourt 
  in December 2007 and June 2009. }

{
\footnotesize
\providecommand{\bysame}{\leavevmode\hbox to3em{\hrulefill}\thinspace}
\providecommand{\MR}{\relax\ifhmode\unskip\space\fi MR }
\providecommand{\MRhref}[2]{%
  \href{http://www.ams.org/mathscinet-getitem?mr=#1}{#2}
}
\providecommand{\href}[2]{#2}

}

\bigskip

\vfill

{\footnotesize %
  \noindent Djalil \textsc{Chafa\"\i}, %
  \noindent\url{mailto:djalil(AT)chafai(DOT)net}
  
  \medskip
  \noindent\textsc{UMR 8050 CNRS Laboratoire d'Analyse et de Math\'ematiques Appliqu\'ees\\
    Universit\'e Paris-Est Marne-la-Vall\'ee \\ 
    5 bd Descartes, Champs-sur-Marne, F-77454 Cedex 2, France.}
  \bigskip
  
  \noindent Florent \textsc{Malrieu},  corresponding author, %
  \url{mailto:florent.malrieu(AT)univ-rennes1(DOT)fr}

  \medskip

  \noindent\textsc{UMR 6625 CNRS Institut de Recherche Math\'ematique de
    Rennes (IRMAR) \\ Universit\'e de Rennes I, Campus de Beaulieu, F-35042
    Rennes \textsc{Cedex}, France.}
  
 \bigskip
  
 \noindent Katy \textsc{Paroux}, %
 \noindent\url{mailto:katy.paroux(AT)univ-fcomte(DOT)fr}

 \medskip

 \noindent\textsc{UMR 6623 CNRS Laboratoire de Math\'ematiques \\
   Universit\'e de Franche-Comt\'e, F-25030 Besan\c con \textsc{Cedex}, France}

 \medskip

 \noindent Present adress   \noindent\textsc{UMR 6625 CNRS Institut de Recherche Math\'ematique de
    Rennes (IRMAR) \\ Universit\'e de Rennes I, Campus de Beaulieu, F-35042
    Rennes \textsc{Cedex}, France.}

\vfill

\begin{flushright}\texttt{Compiled \today.}\end{flushright}

}

\end{document}